\documentclass[leqno,english]{amsart}
\usepackage[T1]{fontenc}
\usepackage[latin9]{inputenc}
\usepackage{color}
\usepackage{babel}
\usepackage{enumitem}
\usepackage{amstext}
\usepackage{amsthm}
\usepackage{amssymb}
\usepackage[unicode=true,pdfusetitle,
 bookmarks=true,bookmarksnumbered=true,bookmarksopen=false,
 breaklinks=false,pdfborder={0 0 0},pdfborderstyle={},backref=page,colorlinks=true]
 {hyperref}
\hypersetup{
 citecolor=blue,linkcolor=red}

\makeatletter
\numberwithin{equation}{section}
\numberwithin{figure}{section}
\numberwithin{table}{section}
\theoremstyle{plain}
\newtheorem{thm}{\protect\theoremname}[section]
\theoremstyle{plain}
\newtheorem{conjecture}[thm]{\protect\conjecturename}
\theoremstyle{plain}
\newtheorem{prop}[thm]{\protect\propositionname}
\theoremstyle{plain}
\newtheorem{lem}[thm]{\protect\lemmaname}
\theoremstyle{plain}
\newtheorem{cor}[thm]{\protect\corollaryname}

\usepackage{tikz,fancyhdr,enumitem}
\usetikzlibrary{scopes,calc,intersections,through,backgrounds,arrows,decorations.pathmorphing,positioning,fit,petri,decorations.markings,matrix}
\usepackage{hyperref,bookmark,dsfont,bbm}
\usepackage[all]{xy}
\hypersetup{linktocpage}


\makeatother

\providecommand{\conjecturename}{Conjecture}
\providecommand{\corollaryname}{Corollary}
\providecommand{\lemmaname}{Lemma}
\providecommand{\propositionname}{Proposition}
\providecommand{\theoremname}{Theorem}

\begin{document}
\title[Circuit-Cocircuit Intersection Conjecture for Intersection Size $k\le6$]{The Circuit-Cocircuit Intersection Conjecture for Intersection Size
$k\le6$}
\author{Jaeho Shin}
\address{Korea Institute for Advanced Study, 85 Hoegiro, Seoul 02455, South
Korea}
\email{shin@kias.re.kr}
\begin{abstract}
Oxley conjectured (1992) that if a matroid has a circuit-cocircuit
intersection of size $k\ge4$, it has a circuit-cocircuit intersection
of size $k-2$. We show that this conjecture holds for $k\le6$.
\end{abstract}

\maketitle

\section{Introduction and Preparations}

We work on finite matroids throughout the paper. Let $M$ be a matroid
on $S=E(M)$ and consider its dual matroid $M^{\ast}$ with $S=E(M^{\ast})$.
A \textbf{cocircuit} of $M$ is a circuit of $M^{\ast}$. If a subset
$X\subseteq S$ is the intersection of a circuit and a cocircuit of
$M$, then $X$ is called a \textbf{circuit-cocircuit intersection}.\smallskip{}

A subset $D\subseteq S$ is a cocircuit of $M$ if and only if $S-D$
is a maximal proper flat of $M$. So, a circuit-cocircuit intersection
is a circuit minus a maximal proper flat, and thus its cardinality
can never be $1$. Oxley conjectured the following.
\begin{conjecture}[Circuit-Cocircuit Intersection Conjecture 1992]
\label{conj:Oxley} If a matroid has a circuit-cocircuit intersection
of size $k\ge4$, it has a circuit-cocircuit intersection of size
$k-2$.
\end{conjecture}

Kingan and Lemos \cite{KL06} showed that the conjecture holds if
the underlying matroid is regular. In this paper, we show that the
conjecture holds for all circuit-cocircuit intersections of size $k\le6$
in an arbitrary matroid. Our proof is based on the following theorems
from \cite{Oxl84}.
\begin{prop}[Oxley 1984]
\label{prop:Oxley} Let $M$ be any matroid.
\begin{enumerate}
\item \label{enu:Oxley1}Every circuit-cocircuit intersection of a minor
of $M$ is a circuit-cocircuit intersection of $M$.
\item \label{enu:Oxley2}Let $X$ be a circuit-cocircuit intersection of
size $k\ge4$. Then, 
\begin{enumerate}
\item \label{enu:Oxley2-1}$M$ has a circuit-cocircuit intersection of
size $4$.
\item \label{enu:Oxley2-2}$M$ has a minor $N$ whose rank and dual rank
are $k-1$ in which $X$ is both a circuit and a cocircuit, and $E(N)-X$
is a rank-$\left(k-2\right)$ flat. Moreover, $N$ is a simple matroid
with $N|_{X}\simeq U_{k-1,k}$ and $N|_{E(N)-X}\simeq U_{k-2,k-2}$
where $U_{r,n}$ denotes the uniform matroid of rank $r$ on $\left\{ 1,\dots,n\right\} $.
\end{enumerate}
\end{enumerate}
\end{prop}

Proposition \ref{prop:Oxley}(\ref{enu:Oxley2-1}) tells that Conjecture
\ref{conj:Oxley} holds for $k=6$. So, we prove the conjecture for
$k=4,5$.\smallskip{}

Note that if $I$ is an independent set of $M$, for any $e\in S$
satisfying that $I\cup\left\{ e\right\} $ is a dependent set, there
is a unique circuit contained in $I\cup\left\{ e\right\} $ which
contains $e$.

\section{\label{sec:Proof}Proof of the Main Theorem}
\begin{lem}
\label{lem:Ce0}Fix an integer $k\ge4$. Let $N$ be a matroid whose
rank and dual rank are $k-1$ with a size-$k$ circuit $X$ that is
also a cocircuit. Then, $N$ is a simple matroid whose ground set
has size $2k-2$ with $N|_{X}\simeq U_{k-1,k}$ and $N|_{E(N)-X}\simeq U_{k-2,k-2}$.
In particular, every circuit has size $\ge3$. For $y\in E(N)-X$,
let 
\[
\mathcal{C}_{y}:=\left\{ \text{the circuits }C\text{ of }N\text{ with }C-X=\left\{ y\right\} \right\} .
\]
Then, 
\begin{enumerate}
\item \label{enu:Ce0-1}$\left|\mathcal{C}_{y}\right|\ge2$.
\item \label{enu:Ce0-2}Let $C_{1}$ and $C_{2}$ be two circuits of $\mathcal{C}_{y}$,
then $\left|X-C_{1}\cap C_{2}\right|\ge2$ and
\[
C_{1}\cup C_{2}=X\cup\left\{ y\right\} .
\]
\item \label{enu:Ce0-3}If $C$ is a circuit of $\mathcal{C}_{y}$ other
than $C_{1}$ and $C_{2}$, then 
\[
X-C_{1}\cap C_{2}\subseteq C-\left\{ y\right\} .
\]
\item \label{enu:Ce0-4}$\left|\mathcal{C}_{y}\right|=2$ if and only if
$C_{1}\cap C_{2}=\left\{ y\right\} $.
\end{enumerate}
\end{lem}

\begin{proof}
Since the rank and the dual rank of $N$ are $k-1$, we have $\left|E(N)\right|=2k-2$.
Then, $N|_{X}\simeq U_{k-1,k}$ since $X$ is a size-$k$ circuit,
and $E(N)-X$ is a maximal proper flat of size $k-2$ since $X$ is
a size-$k$ cocircuit. Moreover, $E(N)-X$ is an independent set with
$N|_{E(N)-X}\simeq U_{k-2,k-2}$.\smallskip{}

Clearly, $N$ has no loops. If $C$ is a size-$2$ circuit, then it
is not contained in $X$ since $X$ is a circuit of size $k\ge4$,
and is not contained in $E(N)-X$ since $E(N)-X$ is an independent
set. Hence, $\left|C\cap X\right|=1$ and $\left|C-X\right|=1$, but
then $C\cap X$ is a circuit-cocircuit intersection of size $1$,
a contradiction. So, $N$ has no size-$2$ circuits. Thus, $N$ is
a simple matroid, and every circuit has size $\ge3$.\smallskip{}

For any size-$\left(k-1\right)$ subset $B\subset X$ which is a base
of $N$, the union $B\cup\left\{ y\right\} $ is dependent. So, there
exists a unique circuit $C\subseteq B\cup\left\{ y\right\} $ which
contains $y$ with $C-X=\left\{ y\right\} $. Note that every circuit
of $\mathcal{C}_{y}$ arises in this way.\smallskip{}

Choose another base $B'\subset X$, then $\left|X-B'\right|=1$ and
$C\cap B'\neq\emptyset$ since $\left|C\right|\ge3$. For any $x\in C\cap B'$,
there is a circuit $C'$ contained in $C\cup X-\left\{ x\right\} =X\cup\left\{ y\right\} -\left\{ x\right\} $.
Then, $C'\neq C$ must contain $y$ since otherwise the circuit $X$
would strictly contain it, a contradiction. Hence, $C'$ is a circuit
of $\mathcal{C}_{y}$, and $\left|\mathcal{C}_{y}\right|\ge2$. Thus,
(\ref{enu:Ce0-1}) is proved.\smallskip{}

Let $C_{1}$ and $C_{2}$ be two circuits of $\mathcal{C}_{y}$. Then,
$\mathcal{C}_{y}$ has a circuit $C$ that is contained in $C_{1}\cup C_{2}-\left\{ y\right\} $.
Then, $C\subseteq C_{1}\cup C_{2}-\left\{ y\right\} \subseteq X$
and therefore $C=X$ and 
\[
C_{1}\cup C_{2}=X\cup\left\{ y\right\} .
\]
 Further, because $X-C_{1}$ and $X-C_{2}$ are distinct nonempty
subsets of $E(N)$, we have $\left|X-C_{1}\cap C_{2}\right|=\left|\left(X-C_{1}\right)\cup\left(X-C_{2}\right)\right|\ge2$.
Thus, (\ref{enu:Ce0-2}) is proved.\smallskip{}

If $C$ is a circuit of $\mathcal{C}_{y}$ other than $C_{1}$ and
$C_{2}$, then $C\cup C_{1}$ contains $X$ by (\ref{enu:Ce0-2})
and $C$ contains $X-C_{1}$. In the same manner, $C$ contains $X-C_{2}$.
Therefore, $C$ contains $\left(X-C_{1}\right)\cup\left(X-C_{2}\right)=X-C_{1}\cap C_{2}$.
Thus, (\ref{enu:Ce0-3}) is proved.\smallskip{}

To prove (\ref{enu:Ce0-4}), suppose that $C_{1}\cap C_{2}=\left\{ y\right\} $.
If $\mathcal{C}_{y}$ has another circuit, say $C$, then $C$ strictly
contains $X-C_{1}\cap C_{2}=X-\left\{ y\right\} =X$ by (\ref{enu:Ce0-3}),
which is a contradiction. Hence, $C_{1}$ and $C_{2}$ are the only
members of $\mathcal{C}_{y}$, and $\left|\mathcal{C}_{y}\right|=2$.

Suppose that $C_{1}\cap C_{2}-\left\{ y\right\} \neq\emptyset$. Then,
for any $x\in C_{1}\cap C_{2}-\left\{ y\right\} \subset X$, there
is a circuit $C$ contained in $C_{1}\cup C_{2}-\left\{ x\right\} =X\cup\left\{ y\right\} -\left\{ x\right\} $,
and so $y\in C$ and $C\in\mathcal{C}_{y}-\left\{ C_{1},C_{2}\right\} $.
Thus, $\left|\mathcal{C}_{y}\right|\ge3$, and (\ref{enu:Ce0-4})
is proved.
\end{proof}
Now, we prove our main theorem.
\begin{thm}
Conjecture \ref{conj:Oxley} holds for $k\le6$.
\end{thm}

\begin{proof}
Let $M$ be a matroid with a size-$k$ circuit-cocircuit intersection
$X$. Then, by Proposition \ref{prop:Oxley}(\ref{enu:Oxley2-2}),
there is a minor $N$ of $M$ whose rank and dual rank are $k-1$
in which $X$ is both a circuit and a cocircuit. So, we can use Lemma
\ref{lem:Ce0}. We prove that $N$ has a circuit-cocircuit intersection
of size $k-2$, then by Proposition \ref{prop:Oxley}(\ref{enu:Oxley1}),
$M$ also has a circuit-cocircuit intersection of size $k-2$ and
we are done. Moreover, it suffices to prove for $k=4,5$ by Proposition
\ref{prop:Oxley}(\ref{enu:Oxley2-1}). For convenience, we write
\[
Y=E(N)-X.
\]

Let $k=4$. Write 
\[
Y=\left\{ y_{1},y_{2}\right\} \quad\text{and}\quad X=\left\{ x_{1},x_{2},x_{3},x_{4}\right\} .
\]
 Since $Y$ is a rank-$2$ flat of size $2$, the union $Y\cup\left\{ x_{1}\right\} $
is a rank-$3$ independent set. So, $Y\cup\left\{ x_{1},x_{2}\right\} $
is a dependent set since the rank of $N$ is $3$, and there is a
circuit $C$ contained in $Y\cup\left\{ x_{1},x_{2}\right\} $. Then,
$\left\{ x_{1},x_{2}\right\} \subset C$ since $\left|C\right|\ge3$
and $\left|Y\cup\left\{ x_{1},x_{2}\right\} \right|=4$. Thus, $C\cap X=\left\{ x_{1},x_{2}\right\} $
which is a circuit-cocircuit intersection of size $2$.\smallskip{}

Let $k=5$. Write 
\[
Y=\left\{ y_{1},y_{2},y_{3}\right\} \quad\text{and}\quad X=\left\{ x_{1},x_{2},x_{3},x_{4},x_{5}\right\} .
\]
 If there is a circuit or a cocircuit $C$ of size $4$ with $\left|C\cap Y\right|=1$,
then $C\cap X$ is a circuit-cocircuit intersection of size $3$,
and we are done. Therefore, we assume that all circuits and cocircuits
$C$ with $\left|C\cap Y\right|=1$ have size $3,5$.\smallskip{}

If $C$ is a (co)circuit of size $3$ with $\left|C\cap Y\right|=1$,
there is a (co)circuit $C'\neq C$ with $C'\cap Y=C\cap Y$ by Lemma
\ref{lem:Ce0}(\ref{enu:Ce0-1}). Then, $X\subset C'\cup C$ by Lemma
\ref{lem:Ce0}(\ref{enu:Ce0-2}), and $X-C\subset C'$. Since $\left|X-C\right|=3$,
we have $\left|C'\right|\ge4$ and $\left|C'\right|=5$ by assumption.
Hence, for any $y\in Y$, there always exists a size-$5$ (co)circuit
$C'$ with $C'\cap Y=\left\{ y\right\} $.\smallskip{}

Let $D$ be a cocircuit of size $5$ with $D\cap Y=\left\{ y_{1}\right\} $,
then $\left|D\cap X\right|=4$ and without loss of generality write
\[
D\cap X=X-\left\{ x_{1}\right\} .
\]
 We show below that there is a circuit $C\in\mathcal{C}_{y_{2}}$
of size $5$ containing $x_{1}$, then $C\cap D$ is a circuit-cocircuit
intersection of size $3$, and the proof will be done.\smallskip{}

Indeed, there is a circuit $C_{1}\in\mathcal{C}_{y_{2}}$ containing
$x_{1}$. If $C_{1}$ has size $5$, we are done. If $C_{1}$ has
size $3$, without loss of generality write 
\[
C_{1}=\left\{ y_{2},x_{1},x_{2}\right\} .
\]
 By (\ref{enu:Ce0-1}) and (\ref{enu:Ce0-2}) of Lemma \ref{lem:Ce0},
there is a circuit $C_{2}\in\mathcal{C}_{y_{2}}$ with 
\[
C_{1}\cup C_{2}=X\cup\left\{ y_{2}\right\} 
\]
 and so $\left|C_{2}\right|\ge4$. Then, $C_{2}$ has size $5$ and
contains either $x_{1}$ or $x_{2}$. If $C_{2}$ contains $x_{1}$,
we are done. If $C_{2}$ contains $x_{2}$, we have 
\[
C_{1}\cap C_{2}=\left\{ y_{2},x_{2}\right\} .
\]
 Then, by (\ref{enu:Ce0-1}) and (\ref{enu:Ce0-4}) of Lemma \ref{lem:Ce0},
there is a circuit $C_{3}\in\mathcal{C}_{y_{2}}-\left\{ C_{1},C_{2}\right\} $,
and $C_{3}-\left\{ y_{2}\right\} $ contains $X-C_{1}\cap C_{2}=X-\left\{ x_{2}\right\} $
by Lemma \ref{lem:Ce0}(\ref{enu:Ce0-3}). Then, $C_{3}$ contains
$X\cup\left\{ y_{2}\right\} -\left\{ x_{2}\right\} $ which has size
$5$, and we have 
\[
C_{3}=X\cup\left\{ y_{2}\right\} -\left\{ x_{2}\right\} 
\]
 which is a size-$5$ circuit containing $x_{1}$. The proof is complete.
\end{proof}

\section{\label{sec:Provision}Towards the Conjecture for Intersection Size
$k\ge7$}

The computation for intersection size $k\ge7$ is very complicated.
But, we expect the following two propositions to play an important
role in solving the conjecture for $k\ge7$. Proposition \ref{prop:Ce1}
investigates the interaction between circuits of $\mathcal{C}_{y}$
and $\mathcal{C}_{y'}$, and Corollary \ref{cor:rank2 circuits} particularly
deals with rank-$2$ circuits of $N$.\smallskip{}

For two sets $A$ and $B$ we denote by $A\,\triangle\,B$ their symmetric
difference 
\[
A\,\triangle\,B=\left(A-B\right)\cup\left(B-A\right)=A\cup B-A\cap B.
\]

\begin{prop}
\label{prop:Ce1}Fix an integer $k\ge4$. Let $N$ be a matroid whose
rank and dual rank are $k-1$ with a size-$k$ circuit $X$ that is
also a cocircuit. Let $C$ and $C'$ be two circuits satisfying that
$\left|C-X\right|=\left|C'-X\right|=1$ and $C\cap C'\neq\emptyset$.
\begin{enumerate}[itemsep=2pt]
\item \label{enu:Ce1-1}If $X\nsubseteq C\cup C'$, then either
\begin{enumerate}[topsep=1pt]
\item $C\cap X$ or $C'\cap X$ contains the other, or
\item there is a circuit $C''$ with $C\,\triangle\,C'\subseteq C''\subset C\cup C'$.
\end{enumerate}
\item \label{enu:Ce1-2}If $C\cap X$ strictly contains $C'\cap X$, then
$C-X\neq C'-X$ and there is a circuit $C''$ satisfying that $C-C'\subseteq C''\subset C\cup C'$
and $C''-X=C\cup C'-X$.
\end{enumerate}
\end{prop}

\begin{proof}
By Lemma \ref{lem:Ce0}, we know that $N$ is a simple matroid with
$\left|E(N)\right|=2k-2$, $N|_{X}\simeq U_{k-1,k}$, and $N|_{E(N)-X}\simeq U_{k-2,k-2}$
where $E(N)-X$ is a maximal proper flat of size $k-2$ which is also
an independent set. By assumption, $C-X=\left\{ y\right\} $ and $C'-X=\left\{ y'\right\} $
for some $y,y'\in E(N)-X$ so that $C\in\mathcal{C}_{y}$ and $C'\in\mathcal{C}_{y'}$.\smallskip{}

To prove (\ref{enu:Ce1-1}) by contrapositive, suppose that $\left(C-C'\right)\cap X\neq\emptyset$,
$\left(C'-C\right)\cap X\neq\emptyset$, and there are no circuits
$C''$ with $C\,\triangle\,C'\subseteq C''\subset C\cup C'$, i.e.
if $C''$ is a circuit contained in $C\cup C'$, then either $\left(C-C'\right)-C''\neq\emptyset$
or $\left(C'-C\right)-C''\neq\emptyset$.

If $y=y'$, then $C,C'\in\mathcal{C}_{y}$ are distinct circuits,
and $X\subset C\cup C'$ by Lemma \ref{lem:Ce0}(\ref{enu:Ce0-2}).

If $y\neq y'$, then $\emptyset\neq C\cap C'\subset X$ and for some
$x_{0}\in C\cap C'$ and $x_{1}\in\left(C-C'\right)\cap X\neq\emptyset$,
there is a circuit $C''$ with $x_{1}\in C''\subseteq C\cup C'-\left\{ x_{0}\right\} $.
Then, $\left\{ y,y'\right\} \cap C''\neq\emptyset$ since otherwise
$C''\subseteq C\cup C'-\left\{ y,y'\right\} \subseteq X$, a contradiction,
and we have $3$ cases: 
\[
y\notin C''\text{ or }y'\notin C''\text{ or }\left\{ y,y'\right\} \subset C''.
\]

If $y\notin C''$, then $C''\in\mathcal{C}_{y'}$, and $X\subset C'\cup C''$
by Lemma \ref{lem:Ce0}(\ref{enu:Ce0-2}), and so $X\subset C\cup C'$.

If $y'\notin C''$, similarly $X\subset C\cup C'$.

If $\left\{ y,y'\right\} \subset C''$, suppose that $\left(C-C'\right)-C''\neq\emptyset$,
then for $x_{1}\in\left(C-C'\right)-C''$ there is a circuit $C_{1}\in\mathcal{C}_{y'}$
with $x_{1}\in C_{1}\subseteq C\cup C'-\left\{ y\right\} $. Then,
$X\subset C_{1}\cup C'$ by Lemma \ref{lem:Ce0}(\ref{enu:Ce0-2})
and $X\subset C\cup C'$. Suppose that $\left(C'-C\right)-C''\neq\emptyset$,
then similarly $X\subset C\cup C'$. Thus, (\ref{enu:Ce1-1}) is proved.\smallskip{}

To prove (\ref{enu:Ce1-2}), suppose that $C'\cap X\subsetneq C\cap X$,
i.e. $\left(C-C'\right)\cap X\neq\emptyset$. Then, $y\neq y'$ since
otherwise $C'\subsetneq C$, and $\left\{ y,y'\right\} =C\cup C'-X$.
For any $x_{2}\in\left(C-C'\right)\cap X$ and $x_{3}\in C'\cap X$,
there is a circuit $C''$ with $x_{2}\in C''\subset C\cup C'-\left\{ x_{3}\right\} $.
Then, $\left\{ y,y'\right\} \cap C''$ is nonempty since otherwise
$C''\subseteq C\cup C'-\left\{ y,y'\right\} \subseteq X$, and we
again have $3$ cases: 
\[
y\notin C''\text{ or }y'\notin C''\text{ or }\left\{ y,y'\right\} \subset C''.
\]

If $y\notin C''$, then $C''\in\mathcal{C}_{y'}$ and $\left(C''\cup C'\right)\cap X=X$
by Lemma \ref{lem:Ce0}(\ref{enu:Ce0-2}), but then $X=\left(C''\cup C'\right)\cap X\subseteq C\cap X\subsetneq X$,
a contradiction.

If $y'\notin C''$, then $y\in C''$ and $C''=\left\{ y\right\} \cup\left(C''\cap X\right)\subset C$,
a contradiction.

Therefore, we have $\left\{ y,y'\right\} \subset C''$ and $C''-X=\left\{ y,y'\right\} =C\cup C'-X$.

Further, since $y'\in C'\cap C''$ and $y\in C''-C'$, there is a
circuit $C_{2}$ satisfying that $y\in C_{2}\subseteq C'\cup C''-\left\{ y'\right\} $,
then $C_{2}\in\mathcal{C}_{y}$ and actually $C_{2}=C$ since $C'\cup C''-\left\{ y'\right\} \subseteq C$.
Therefore $C-C'=C_{2}-C'\subseteq C''$, and (\ref{enu:Ce1-2}) is
proved. The proof is complete.
\end{proof}
\begin{cor}
\label{cor:rank2 circuits}Every rank-$2$ circuit of $N$ is a circuit
of $\mathcal{C}_{y}$ for some $y\in E(N)-X$, and is also a flat.
If $k\ge5$, for any two rank-$2$ circuits $C$ and $C'$, either 
\begin{enumerate}
\item $C\cap C'=\emptyset$, or
\item $C\cap C'$ is a singleton contained in $X$ and $C\,\triangle\,C'$
is a circuit of size $4$.
\end{enumerate}
\end{cor}

\begin{proof}
Let $C$ be a rank-$2$ circuit, then $C\nsubseteq X$ since $X$
is a circuit of size $k\ge4$, and $C-X\neq\emptyset$. Since $E(N)-X$
is an independent set, $C\nsubseteq E(N)-X$ and $C\cap X\neq\emptyset$.
Since $C\cap X$ is a circuit-cocircuit intersection, we have $\left|C\cap X\right|\ge2$.
Because $C$ is a rank-$2$ circuit, $\left|C\cap X\right|=2$ and
$\left|C-X\right|=1$, and let $C-X=\left\{ y\right\} $, then $C\in\mathcal{C}_{y}$.
Moreover, $C$ is a rank-$2$ flat. Indeed, for all $y'\in E(N)-C\cup X=E(N)-X\cup\left\{ y\right\} $,
we have $r(C\cup\left\{ y'\right\} )=3$ since otherwise $\left\{ y,y'\right\} =\left(C-X\right)\cup\left\{ y'\right\} \subset E(N)-X$
would be a rank-$2$ flat and contain $C$. Also, for all $x\in X-C$,
we have $r(C\cup\left\{ x\right\} )=3$ since $\left(C\cap X\right)\cup\left\{ x\right\} \subset X$
is a size-$3$ independent set. This proves that $C$ is a flat.\smallskip{}

Suppose $k\ge5$. Let $C$ and $C'$ be two  rank-$2$ circuits, which
are also rank-$2$ flats, then we have $\left|C\cap X\right|=\left|C'\cap X\right|=2$
and $\left|C-X\right|=\left|C'-X\right|=1$, and moreover $C\cap C'$
is a flat of rank $\le1$. Because $N$ is simple, either $C\cap C'=\emptyset$
or $\left|C\cap C'\right|=1$. Suppose $C\cap C'\neq\emptyset$, then
$\left|C\cap C'\right|=1$. Since $\left|\left(C\cup C'\right)\cap X\right|\le4<5\le k=\left|X\right|$,
we have $X\nsubseteq C\cup C'$ and $C-X\neq C'-X$ by Lemma \ref{lem:Ce0}(\ref{enu:Ce0-2}).
In particular, $C\cap C'$ is a singleton contained in $X$. Note
that $C\cap X\nsubseteq C'\cap X$ and $C\cap X\nsupseteq C'\cap X$.
Then, by Proposition \ref{prop:Ce1}(\ref{enu:Ce1-1}), there is a
circuit $C''$ satisfying that $C\,\triangle\,C'\subseteq C''\subset C\cup C'$.
But, $4=\left|C\,\triangle\,C'\right|\le\left|C''\right|\le\left|C\cup C'\right|-1=4$,
and thus $C''=C\,\triangle\,C'$ which is a circuit of size $4$.
\end{proof}

\subsection*{Acknowledgements}

The author would like to thank James Oxley for encouraging him to
publish this result. He is also grateful to JongHae Keum for supporting
him by the grant (NRF-2019R1A2C3010487) of the National Research Foundation
funded by the Korean government.

\end{document}